\documentclass[11pt]{amsart}
\usepackage{amssymb}

\makeatletter
\newcommand{\sumprime}{\if@display\sideset{}{'}\sum%
            \else\sum'\fi}
\makeatother

\begin{document}

\numberwithin{equation}{section}

\newtheorem{theorem}{Theorem}[section]
\newtheorem{proposition}[theorem]{Proposition}
\newtheorem{conjecture}[theorem]{Conjecture}
\def\theconjecture{\unskip}
\newtheorem{corollary}[theorem]{Corollary}
\newtheorem{lemma}[theorem]{Lemma}
\newtheorem{observation}[theorem]{Observation}
\newtheorem{definition}{Definition}
\numberwithin{definition}{section} 
\newtheorem{remark}{Remark}
\def\theremark{\unskip}
\newtheorem{question}{Question}
\def\thequestion{\unskip}
\newtheorem{example}{Example}
\def\theexample{\unskip}
\newtheorem{problem}{Problem}

\def\vvv{\ensuremath{\mid\!\mid\!\mid}}
\def\intprod{\mathbin{\lr54}}
\def\reals{{\mathbb R}}
\def\integers{{\mathbb Z}}
\def\N{{\mathbb N}}
\def\complex{{\mathbb C}\/}
\def\dist{\operatorname{dist}\,}
\def\spec{\operatorname{spec}\,}
\def\interior{\operatorname{int}\,}
\def\trace{\operatorname{tr}\,}
\def\cl{\operatorname{cl}\,}
\def\essspec{\operatorname{esspec}\,}
\def\range{\operatorname{\mathcal R}\,}
\def\kernel{\operatorname{\mathcal N}\,}
\def\dom{\operatorname{Dom}\,}
\def\linearspan{\operatorname{span}\,}
\def\lip{\operatorname{Lip}\,}
\def\sgn{\operatorname{sgn}\,}
\def\Z{ {\mathbb Z} }
\def\e{\varepsilon}
\def\p{\partial}
\def\rp{{ ^{-1} }}
\def\Re{\operatorname{Re\,} }
\def\Im{\operatorname{Im\,} }
\def\dbarb{\bar\partial_b}
\def\eps{\varepsilon}
\def\O{\Omega}
\def\Lip{\operatorname{Lip\,}}

\def\Hs{{\mathcal H}}
\def\E{{\mathcal E}}
\def\scriptu{{\mathcal U}}
\def\scriptr{{\mathcal R}}
\def\scripta{{\mathcal A}}
\def\scriptc{{\mathcal C}}
\def\scriptd{{\mathcal D}}
\def\scripti{{\mathcal I}}
\def\scriptk{{\mathcal K}}
\def\scripth{{\mathcal H}}
\def\scriptm{{\mathcal M}}
\def\scriptn{{\mathcal N}}
\def\scripte{{\mathcal E}}
\def\scriptt{{\mathcal T}}
\def\scriptr{{\mathcal R}}
\def\scripts{{\mathcal S}}
\def\scriptb{{\mathcal B}}
\def\scriptf{{\mathcal F}}
\def\scriptg{{\mathcal G}}
\def\scriptl{{\mathcal L}}
\def\scripto{{\mathfrak o}}
\def\scriptv{{\mathcal V}}
\def\frakg{{\mathfrak g}}
\def\frakG{{\mathfrak G}}

\def\ov{\overline}

\thanks{Research supported by the Key Program of NSFC No. 11031008.}

\address{Department of Mathematics, Tongji University, Shanghai, 200092, China}
\email{boychen@tongji.edu.cn}
\address{Department of Mathematics, Tongji University, Shanghai, 200092, China}
\email{1113xuwang@tongji.edu.cn}

\title{Holomorphic maps with large images}
 \author{Bo-Yong Chen and Xu Wang}
\date{}
\maketitle

\medskip

\centerline{\it Dedicated to the memory of Shoshichi Kobayashi}

\medskip

\begin{abstract}  We show that each pseudoconvex domain $\Omega\subset {\mathbb C}^n$ admits a holomorphic map $F$ to ${\mathbb C}^m$ with $|F|\le C_1 e^{C_2 \hat{\delta}^{-6}}$, where $\hat{\delta}$ is the minimum of the boundary distance and $(1+|z|^2)^{-1/2}$, such that every boundary point is a Casorati-Weierstrass point of $F$. Based on this fact, we introduce a new anti-hyperbolic concept --- universal dominability. We also show that for each $\alpha>6$ and each pseudoconvex domain $\Omega\subset {\mathbb C}^n$, there is a holomorphic function $f$ on ${\Omega}$ with $|f|\le C_\alpha e^{C_\alpha' \hat{\delta}^{-\alpha}}$, such that every boundary point is a Picard point of $F$. Applications to the construction of holomorphic maps of a given domain onto some ${\mathbb C}^m$ are given.
\bigskip


\smallskip

\end{abstract}

\section{Introduction}

The isolated singularities of holomorphic functions in one complex variable were completely clarified through the classical theorems of Casorati-Weierstrass and Picard, which also mark the beginning of two great theories: the theory of cluster sets and the value distribution theory of Nevanlinna. The case of non-isolated singularities is much more complicated. Weierstrass is also the first person who considered the problem of finding a holomorphic function $f$ on a given domain $\Omega\subset {\mathbb C}$ such that each point of $\partial \Omega$ is a singularity of $f$. For domains $\Omega\subset {\mathbb C}^n$, such a property is equivalent to say that $\Omega$ is pseudoconvex, in view of the solution of Levi's problem by Oka, Bremermann and Norguet. Jarnicki and Pflug (cf. \cite{JarnickiPflugExtension}, p. 181) showed furthermore that there exists a holomorphic function $f$ on $\Omega$ such that $\hat{\delta}^k f$ is bounded on $\Omega$ for any $k>6n$, and each point of $\partial\Omega$ is a singularity of $f$. Here $\delta(z):={\rm dist\,}(z,\Omega^c)$ and $\hat{\delta}(z)=\min\{\delta(z),(1+|z|^2)^{-1/2}\}$.

In view of the theorem of Casorati-Weierstrass, we introduce the following

\begin{definition}[compare \cite{CollingwoodLohwater}] Consider a domain $\Omega\subset {\mathbb C}^n$ and a holomorphic map $F:\Omega\rightarrow {\mathbb C}^m$. Let\/ ${\mathbb C}_\infty$ be the Riemann sphere.
A point $p\in \partial \Omega\cup \{\infty\}$ is said to be a Casorati-Weierstrass point of $F$ if the cluster set $C_\Omega(F,p)$ of $F$ at $p$ coincides with the Osgood space ${\mathbb C}_\infty^m:=({\mathbb C}_\infty)^m$, where
      \begin{equation*}
        C_\Omega(F,p)=\bigcap_{r>0} \overline{F(\Omega\cap B(p,r))},
      \end{equation*}
      $B(p,r)$ being the ball with center $p$ and radius $r$ $($we set $B(\infty,r)=\{z\in {\mathbb C}^n:|z|>r\})\,$.
  \end{definition}

Let $\mathcal{O}(\Omega,{\mathbb C}^m)$ be the space of all holomorphic maps $F:\Omega\rightarrow {\mathbb C}^m$ with the topology of uniform convergence on compact subsets. For any $\alpha>0$, we define\/ ${\rm CW}_\alpha(\Omega,{\mathbb C}^m)$ to be the set of all maps $F\in \mathcal{O}(\Omega,{\mathbb C}^m)$ whose sets of Casorati-Weierstrass points coincide with $\partial \Omega\cup \{\infty\}$ ($\partial \Omega$ if $\Omega$ is bounded), and there are constants $C_1,C_2>0$ such that
$$
 |F(z)|\le C_1 e^{C_2 \hat{\delta}(z)^{-\alpha}}, \ \ \ z\in \Omega.
$$

\begin{theorem} If\/ $\Omega\subset {\mathbb C}^n$ is a pseudoconvex domain, then\/ ${\rm{CW}}_6(\Omega,{\mathbb C}^m)$ is nonempty. Furthermore, if\/ $\Omega$ is bounded, then\/ ${\rm{CW}}_2(\Omega,{\mathbb C}^m)$ is nonempty.
\end{theorem}

In general, the growth order of $F$ can not be improved into polynomial growth in $\hat{\delta}^{-1}$. To see this, simply take $\Omega={\mathbb C}$ and note that every entire function with polynomial growth in $1+|z|$ has to be a polynomial. Yet it is still interesting to know the optimal number $\alpha$ such that  ${\rm{CW}}_\alpha(\Omega,{\mathbb C}^m)\neq \emptyset$.

Recall that a complex space $M$ is said to be dominated by another complex space $M'$ if there is a dominant morphism $F:M'\rightarrow M$, i.e., a holomorphic map with dense image. Throughout this paper, a complex space is always reduced. The point is that the pullback of $F$ defines an injective homomorphism $F^\ast: {\mathcal M}(M)\rightarrow {\mathcal M}(M')$ between fields of meromorphic functions on $M$ and $M'$, respectively. We remark that the notion of "dominated" adapted here comes from algebraic geometry (see e.g., \cite{HartshorneBook}\,), which is different from the more commonly used notion of dominability where one requires the existence of a map with surjective derivative at some point. An immediate consequence of Theorem 1.1 is that ${\mathbb C}^m$ is dominated by any pseudoconvex domain in ${\mathbb C}^n$. If an irreducible complex space $M$ admits a nonconstant $f$, then $f(M)$ is a domain in ${\mathbb C}$, so that ${\mathbb C}^m$ is also dominated by $M$. It is reasonable to introduce the following

\begin{definition} A complex space is said to be\/ {\it universally dominated} if it is dominated by any irreducible complex space with a nonconstant holomorphic function.
\end{definition}

This concept is inspired by Winkelmann \cite{Winkelmann05} (see also \cite{ForstnericWinkelmann}), who proved that each irreducible complex space is dominated by any irreducible complex space admitting a nonconstant bounded holomorphic function, and called the latter a universally dominating space. There have been a few interesting examples of universally dominated spaces, especially for complex surfaces (cf. \cite{BuzzardLu}, \cite{Winkelmann05}). In this paper, we obtain some basic properties of universally
dominated spaces, together with a number of new examples, including all Hopf manifolds for instance. It is easy to see that universally dominated spaces are anti-hyperbolic, i.e., the Kobayashi pseudodistance vanishes identically. Nevertheless, universal dominability is a bimeromorphic invariant, which implies in particular that all Kummer manifolds are universally dominated.

There are several other important anti-hyperbolic manifolds, for instance, elliptic manifolds introduced by Gromov \cite{GromovOka} or Oka manifolds introduced by Forstneri$\check{\rm c}$ \cite{ForstnericBook}. It is known from  Forstneri$\check{\rm c}$-Ritter \cite{ForstnericRitter} that every Oka manifold is dominated by ${\mathbb C}$, thus it is universally dominated in view of Proposition 4.1. On the other hand, Andrist-Wold \cite{AndristWold} showed that the complement of the unit closed ball in ${\mathbb C}^n$ for $n\ge 3$ fails to be elliptic (or even subelliptic), yet it is universally dominated (see Example 4.3). It is not known whether there exists a universally dominated manifold which is not Oka.

  A point $p\in \partial \Omega\cup \{\infty\}$ is said to be a Picard point of a function $f\in {\mathcal O}(\Omega)$ if $f$ assumes infinitely often in any neighborhood of $p$ all complex numbers with at most one exception. Similar as Theorem 1.5 in \cite{ForstnericGlobevnik}, one may define a Picard point for a holomorphic map $F:\Omega\rightarrow {\mathbb C}^m$ if it is a Picard point for any function of the form $P(F)$ where $P$ is a nonconstant polynomial on ${\mathbb C}^m$. For each $\alpha>0$, we define ${\rm P}_{\alpha}(\Omega)$  to be the set of holomorphic functions $f$ on $\Omega$ whose sets of Picard points coincide with $\partial \Omega\cup \{\infty\}$ ($\partial \Omega$ if $\Omega$ is bounded), and there are constants $C_1,C_2>0$ such that
$$
 |f(z)|\le C_1 e^{C_2 \hat{\delta}(z)^{-\alpha}}, \ \ \ z\in \Omega.
$$

\begin{theorem} If\/ $\Omega\subset {\mathbb C}^n$ is a pseudoconvex domain, then\/ ${\rm P}_\alpha(\Omega)$ is nonempty for any $\alpha>6$. Furthermore, if\/ $\Omega$ is bounded, then\/ ${\rm P}_\alpha(\Omega)$ is nonempty for any $\alpha>2$.
\end{theorem}

For each bounded domain $\Omega\subset {\mathbb C}^n$, we define the Bergman space to be
$$
   A^2_\alpha (\Omega):=\left\{f\in {\mathcal O}(\Omega):\int_\Omega |f|^2 \delta^\alpha<\infty\right\},\ \ \ \alpha>-1.
$$
We also write $A^2(\Omega)$ for $A^2_0(\Omega)$ for the sake of simplicity.

\begin{proposition} Let $\Omega\subset {\mathbb C}^n$ be a bounded pseudoconvex domain with $n\ge 2$.
 \begin{enumerate}
  \item Let $p\in \partial \Omega$ be a point which admits an inner ball, i.e., a ball contained in $\Omega$ whose boundary touches $\partial \Omega$ only at $p$. Then for each $\alpha>0$ there is a function in $A^2_\alpha(\Omega)$ with $p$ as Picard point.
  \item    Furthermore, if there is a negative plurisubharmonic (psh) function $\psi$ on $\Omega$ satisfying $C_1 |z-p|^{\beta_1}\le -\psi(z)\le C_2 |z-p|^{\beta_2}$ for some positive constants $C_2,C_2,\beta_1,\beta_2$, then there is a holomorphic function in $A^2(\Omega)$ with $p$ as Picard point.
      \end{enumerate}
\end{proposition}

We remark that the first conclusion fails when $n=1$ or $\alpha=0$. To see this, consider simply the domain ${\mathbb D}^\ast={\mathbb D}\backslash \{0\}$ or ${\mathbb D}\times {\mathbb D}^\ast$, where ${\mathbb D}$ is the unit disc in ${\mathbb C}$.  If $f\in A^2_\alpha({\mathbb D}^\ast)$ with $\alpha>0$, then $z=0$ would be either a removable singularity or a pole of $f$, which is not a Picard point. If $f\in A^2({\mathbb D}\times {\mathbb D}^\ast)$, then $f\in A^2({\mathbb D}^2)$ so that the origin is not a Picard point for $f$. Nevertheless, there is an inner ball at the origin for both cases. Note also that the condition in the second conclusion holds when $\Omega$ is of finite type in the sense of D'Angelo (cf. \cite{Catlin87}).

The proofs of Theorem 1.2 and Proposition 1.3 rely heavily on Lindel\"of type principles due to Cima-Krantz \cite{CimaKrantz} and Hahn \cite{HahnNormal}, which also provide new ways of constructing holomorphic maps from a given domain\/ {\it onto} some ${\mathbb C}^m$. For instance, we are able to show

\begin{proposition}
Let $\Omega$ be a domain in ${\mathbb C}^n$. Suppose there is a point $p\in \partial \Omega$ such that $\Omega$ contains a cone $\Lambda_p$ with vertex at $p$, and there is a supporting complex hyperplane of $\Omega$ at $p$. Then there is a holomorphic map from $\Omega$ onto ${\mathbb C}^n$ which is locally biholomorphic.
\end{proposition}

In particular, every convex domain or bounded domain with Lipschitz boundary in ${\mathbb C}^n$ admits a holomorphic map onto ${\mathbb C}^n$ which is locally biholomorphic. It remains open whether an arbitrary domain in ${\mathbb C}^n$, $n\ge 2$, admits a holomorphic map onto ${\mathbb C}^n$ which is locally biholomorphic. A related and even more interesting question is whether the Runge approximation theorem holds in the class of locally biholomorphic maps from a domain in ${\mathbb C}^n$ to ${\mathbb C}^n$ (not necessary onto). Fornaess and Stout proved that the unit ball or polydisc admits a locally biholomorphic, finite holomorphic map onto\/ {\it any} complex manifold of same dimension (cf. \cite{FornaessStout77}, \cite{FornaessStout82}). Slightly later, L${\o}$w \cite{Low83} showed that each bounded domain with $C^2-$boundary in ${\mathbb C}^n$ admits a holomorphic map onto any complex manifold of dimension $n$. His map is neither locally biholomorphic nor finite, however.

 It is proved in \cite{ForstnericGlobevnik} that, if $F:{\mathbb D}\rightarrow {\mathbb C}^2$ is a proper holomorphic map, then (almost) all points of the circle are Casorati-Weierstrass (or even Plessner) points, for any functions $P(f_1,f_2)$ where $P$ is a rational function on ${\mathbb C}^2$.

 \section{Preliminaries}

 Let $\Omega$ be a pseudoconvex domain in ${\mathbb C}^n$. Put $\delta(z):={\rm dist\,}(z,\Omega^c)$ and
 $$
 \hat{\delta}(z)=\min\{\delta(z),(1+|z|^2)^{-1/2}\}.
 $$
By Oka's theorem, $-\log \hat{\delta}$ is a psh exhaustion function on $\Omega$. Since $|\delta(z)-\delta(w)|\le |z-w|$ and
 $$
 \left|(1+|z|^2)^{-1/2}-(1+|w|^2)^{-1/2}\right|\le |z-w|,
 $$
 it follows that
 $
 |\hat{\delta}(z)-\hat{\delta}(w)|\le 2|z-w|.
 $
 Thus
 \begin{equation}
 \hat{\delta}(z)\asymp \hat{\delta}(w),\ \ \ \ \ \   \forall\,z\in B\left(w,c\hat{\delta}(w)\right), \ c<1/2,
 \end{equation}
  where $A\asymp B$ means $C_1A\le B\le C_2A$ for two constants $C_1,C_2$ depending only on $n,c$, and $B(z,r)$ stands for the ball with center $z$ and radius $r$.
  We have the following elementary lemma:

  \begin{lemma}
  If $f$ is a holomorphic function on $\Omega$ such that
  $$
  \int_\Omega |f|^2 e^{-a (1+|z|^2)\hat{\delta}^{-\alpha}}<\infty
  $$
  for some $a,\alpha>0$, then there are constants $C_1,C_2>0$ such that
$$
 |f(z)|\le C_1 e^{C_2 \hat{\delta}(z)^{-\alpha-2}}, \ \ \ z\in \Omega.
$$
  \end{lemma}

  \begin{proof}
  Note that
  \begin{eqnarray*}
  \int_\Omega |f|^2 e^{-a (1+|z|^2)\hat{\delta}^{-\alpha}} & \ge &  \int_\Omega |f|^2 e^{-a \hat{\delta}^{-\alpha-2}}\ge  \int_{B\left(z,\hat{\delta}(z)/3\right)} |f|^2 e^{-a \hat{\delta}^{-\alpha-2}}\\
  & \ge & C |f(z)|^2 \hat{\delta}(z)^{2n}e^{-a'\hat{\delta}(z)^{-\alpha-2}}\\
  & \ge & C |f(z)|^2 e^{-a''\hat{\delta}(z)^{-\alpha-2}}
  \end{eqnarray*}
  where $a''\gg a'\gg a$, and the third inequality follows from (2.1) and the sub-mean-value inequality.
  \end{proof}

The key observation for proving Theorem 1.1 and Theorem 1.2 is the following

 \begin{proposition}
 \begin{enumerate}
 \item Let $\{q_k\}_{k=1}^\infty$ be a sequence of points in $\Omega$ such that the balls $B(q_k,c_n \hat{\delta}(q_k))$ are mutually disjoint for  $c_n\ll1$.
 Let $\{c^k\}_{k\ge 1}\subset {\mathbb C}^m$ satisfy
 $$
 \sum_{k=1}^\infty |c^k|^2 e^{-a\,\hat{\delta}(q_k)^{-4}}<\infty,\ \ \ \forall\,a>0.
 $$
  Then there is a map $F\in {\mathcal O}(\Omega,{\mathbb C}^m)$ such that $F(q_k)=c^k$ for any $k$, and
  $$
 |F(z)|\le C_1 e^{C_2 \hat{\delta}(z)^{-6}}, \ \ \ z\in \Omega.
$$
  \item  Let $\{q_k\}_{k=1}^\infty$ be a sequence of points in $\Omega$ such that the balls
  $$
  B\left(q_k,c_n \hat{\delta}(q_k)|\log(\hat{\delta}(q_k)/2)|^{-1}\right)
  $$
   are mutually disjoint for  $c_n\ll1$.
 Let $\alpha>4$ and $\{c^k\}_{k\ge 1}\subset {\mathbb C}^m$ satisfy
 $$
 \sum_{k=1}^\infty |c^k|^2 e^{-a\,\hat{\delta}(q_k)^{-\alpha}}<\infty,\ \ \ \forall\,a>0.
 $$
  Then there is a map $F\in {\mathcal O}(\Omega,{\mathbb C}^m)$ such that $F(q_k)=c^k$ for any $k$, and
  $$
 |F(z)|\le C_1 e^{C_2 \hat{\delta}(z)^{-\alpha-2}}, \ \ \ z\in \Omega.
$$
 \end{enumerate}
 \end{proposition}

 \begin{proof}
 Put first
 $$
 \psi(z)=(1+|z|^2)\hat{\delta}(z)^{-4}=\exp(-4\log \hat{\delta}(z)+\log(1+|z|^2))
 $$
 for case (1), and
  $
 \psi(z)=(1+|z|^2)\hat{\delta}(z)^{-\alpha}
 $
 for case (2).
 Clearly,
 \begin{equation}
 i\partial\bar{\partial}\psi\ge \psi i\partial\bar{\partial} \log(1+|z|^2)\ge \frac{1+|z|^2}{\hat{\delta}^4}\frac{i\partial\bar{\partial}|z|^2}{(1+|z|^2)^2}\ge \frac{i\partial\bar{\partial}|z|^2}{\hat{\delta}^2}
  \end{equation}
  for case (1) and
    \begin{equation}
 i\partial\bar{\partial}\psi\ge \frac{i\partial\bar{\partial}|z|^2}{\hat{\delta}^{\alpha-2}}
 \end{equation}
 for case (2).
Put $r_k=c_n \hat{\delta}(q_k)$ for case (1), and $r_k=c_n \hat{\delta}(q_k)|\log (\hat{\delta}(q_k)/2)|^{-1}$ for case $(2)$. Let $0\le \chi\le 1$ be a $C^\infty$ function such that $\chi|_{(-\infty,1/2)}=1$ and $\chi|_{(1,\infty)}=0$. By (2.2) and (2.3), there is a sufficiently large constant $C$ such that $i\partial\bar{\partial}\varphi\ge \hat{\delta}^{-2}i\partial\bar{\partial}|z|^2$ for both cases, where
$$
\varphi=C\psi+2n\sum_{k=1}^\infty \chi(|z-q_k|/r_k)\log |z-q_k|/r_k.
$$
Put
$$
v=(v_1,\cdots,v_m):=\sum_{k=1}^\infty c^k \bar{\partial}(\chi (|z-q_k|/r_k)).
$$
Clearly, $v$ is a $C^\infty$ $\bar{\partial}-$closed $m-$vector valued $(0,1)$ form on $\Omega$ and satisfies
\begin{eqnarray*}
\int_\Omega |v|^2_{i\partial\bar{\partial}\varphi} e^{-\varphi} &:=& \int_\Omega \left(|v_1|^2_{i\partial\bar{\partial}\varphi}+\cdots+|v_m|^2_{i\partial\bar{\partial}\varphi}\right) e^{-\varphi}\\
& \le & C' \sum_{k=1}^\infty |c^k|^2\int_{B(q_k,r_k)\backslash B(q_k,r_k/2)}|z-q_k|^{-2n} e^{-C(1+|z|^2)\hat{\delta}^{-4}}\\
& \le & C' \sum_{k=1}^\infty |c^k|^2 e^{-c_n\hat{\delta}(q_k)^{-4}}<\infty
\end{eqnarray*}
for case (1), where $c_n\ll1$, and similarly,
$$
\int_\Omega |v|^2_{i\partial\bar{\partial}\varphi} e^{-\varphi}\le C' \sum_{k=1}^\infty |c^k|^2 e^{-c_n\hat{\delta}(q_k)^{-\alpha}}<\infty
$$
for case (2). Applying H\"ormander's $L^2-$estimates for the $\bar{\partial}-$equation (with values in the trivial $m-$vector bundle) (cf. \cite{Hormander65}, see also \cite{AndreottiVesentini65}, \cite{Demailly82}), we may solve the equation $\bar{\partial}u=v$ on $\Omega$ such that
$$
\int_\Omega |u|^2 e^{-\varphi}\le \int_\Omega |v|^2_{i\partial\bar{\partial}\varphi} e^{-\varphi}<\infty.
$$
Put
$$
F(z)=\sum_{k=1}^\infty c^k \chi (|z-q_k|/r_k)-u(z).
$$
Clearly, we have $F\in {\mathcal O}(\Omega,{\mathbb C}^m)$, $F(q_k)=c^k$ for each $k$, and
$
\int_\Omega |F|^2 e^{-C (1+|z|^2)\hat{\delta}^{-4}}<\infty
$
for case (1), and
$
\int_\Omega |F|^2 e^{-C (1+|z|^2)\hat{\delta}^{-\alpha}}<\infty
$
for case (2). Combining with Lemma 2.1, the conclusion immediately follows.
\end{proof}

Choosing
 $
 \psi(z)=(1+|z|^2)\hat{\delta}(z)^{-2}
 $
 for case (1), and
  $
 \psi(z)=(1+|z|^2)\hat{\delta}(z)^{-\alpha}
 $
 with $\alpha>2$ for case (2) in the proof of the previous proposition, we can prove similarly the following

 \begin{proposition} Let $\Omega$  be a bounded pseudoconvex domain in ${\mathbb C}^n$.
 \begin{enumerate}
 \item Let $\{q_k\}_{k=1}^\infty$ be a sequence of points in $\Omega$ such that the balls $B(q_k,c_n \hat{\delta}(q_k))$ are mutually disjoint for  $c_n\ll1$.
 Let $\{c^k\}_{k\ge 1}\subset {\mathbb C}^m$ satisfy
 $$
 \sum_{k=1}^\infty |c^k|^2 e^{-a\,\hat{\delta}(q_k)^{-2}}<\infty,\ \ \ \forall\,a>0.
 $$
  Then there is a map $F\in {\mathcal O}(\Omega,{\mathbb C}^m)$ such that $F(q_k)=c^k$ for any $k$, and
  $$
 |F(z)|\le C_1 e^{C_2 \hat{\delta}(z)^{-2}}, \ \ \ z\in \Omega.
$$
  \item  Let $\{q_k\}_{k=1}^\infty$ be a sequence of points in $\Omega$ such that the balls
  $$
  B\left(q_k,c_n \hat{\delta}(q_k)|\log(\hat{\delta}(q_k)/2)|^{-1}\right)
  $$
   are mutually disjoint for  $c_n\ll1$.
 Let $\alpha>2$ and $\{c^k\}_{k\ge 1}\subset {\mathbb C}^m$ satisfy
 $$
 \sum_{k=1}^\infty |c^k|^2 e^{-a\,\hat{\delta}(q_k)^{-\alpha}}<\infty,\ \ \ \forall\,a>0.
 $$
  Then there is a map $F\in {\mathcal O}(\Omega,{\mathbb C}^m)$ such that $F(q_k)=c^k$ for any $k$, and
  $$
 |F(z)|\le C_1 e^{C_2 \hat{\delta}(z)^{-\alpha}}, \ \ \ z\in \Omega.
$$
 \end{enumerate}
 \end{proposition}

We also need the following

\begin{lemma}
Let $\{q_k\}_{k=1}^\infty$ be as in Proposition 2.2.  Then for any $a,\alpha>0$, we have $\sum_{k=1}^\infty e^{-a\hat{\delta}(q_k)^{-\alpha}}<\infty$.
\end{lemma}

\begin{proof}
For case (1), we have
\begin{eqnarray*}
\infty> \int_{{\mathbb C}^n} e^{-a (1+|z|^2)^{\alpha''/2}} & \ge & \int_\Omega e^{-a\hat{\delta}^{-\alpha''}}\ge \sum_{k=1}^\infty \int_{B(q_k,r_k)} e^{-a\hat{\delta}^{-\alpha''}}\\
& \ge & C  \sum_{k=1}^\infty \hat{\delta}(q_k)^{2n} e^{-a\hat{\delta}(q_k)^{-\alpha'}}\\
& \ge & C  \sum_{k=1}^\infty  e^{-a\hat{\delta}(q_k)^{-\alpha}}
\end{eqnarray*}
provided $\alpha''\ll\alpha'\ll \alpha$. Case (2) is similar.
\end{proof}

In concluding this section, we recall a classical result from Oka-Cartan's theory.  Let $\Omega$ be a pseudoconvex domain in ${\mathbb C}^n$ and $K$ a compact subset of $\Omega$. Let $\rho$ be a $C^\infty$ strictly psh exhaustion function on $\Omega$ such that $K\subset \{\rho<0\}$.

\begin{proposition} Let $\{q_k\}_{k=1}^\infty$ be a discrete sequence of points in $\{\rho\ge 1\}$. Let $F_0\in {\mathcal O}(\Omega,{\mathbb C}^m)$, and $\{c^k\}_{k=1}^\infty$ be a sequence of points in ${\mathbb C}^m$.   Then for any $\varepsilon>0$ there is a  map $F \in {\mathcal O}(\Omega, {\mathbb C}^m)$ such that $F(q_k)=c^k$ for each $k$ and $
|F-F_0|_K:=\sup_K|F-F_0|<\varepsilon$.
\end{proposition}

\section{Casorati-Weierstrass points}

\begin{proof}[Proof of Theorem 1.1.]
 Let ${\mathcal F}$ be a Whitney decomposition of $\Omega$, i.e., ${\mathcal F}$ is a sequence of cubes $\{Q_1,Q_2,\cdots,Q_k,\cdots\}$ such that

 (1) $\Omega=\cup Q_k$.

 (2) The interiors $Q_k^\circ$ of $Q_k$ are mutually disjoint.

 (3) ${\rm diam\,}(Q_k)\le {\rm dist\,}(Q_k,\Omega^c)\le 4\, {\rm diam\,}(Q_k)$.\\
 (see e.g., \cite{SteinSingularIntegral}, p. 167).

 Let $q_k$ denote the center of $Q_k$ and $r_k=c_n\hat{\delta}(q_k)$ where $c_n$ is sufficiently small so that $B(q_k,r_k)\subset Q_k^\circ$.
We take first a sufficiently large integer $N$ such that the following sets
 $$
 E_\mu:=\bigcup_{Q_j\cap \{2^{-N\mu}\le \hat{\delta}\le 2^{-N\mu+1}\}\neq \emptyset} B(q_j,r_j),\ \ \ \mu\ge 1,
 $$
 are mutually disjoint. Let $\{\zeta^k\}_{k=1}^\infty$ be a sequence of complex $m-$vectors which are dense in ${\mathbb C}_\infty^m$. Note that $\sum_{j=1}^\infty e^{-\hat{\delta}(q_j)^{-3}}<\infty$, in view of Lemma 2.4. Thus there is an integer $\mu_1$ such that
 $$
|\zeta^1|^2 \sum_{\mu=\mu_1}^\infty \sum_{q_j\in E_{2\mu-1}} e^{-\hat{\delta}(q_j)^{-3}} <1/2.
$$
We renumerate the sequence of positive even numbers by $\{\nu^1_{\mu}\}_{\mu=1}^\infty$ and take $\mu_2>0$ such that
$$
|\zeta^2|^2 \sum_{\mu=\mu_2}^\infty \sum_{q_j\in E_{\nu^1_{2\mu-1}}} e^{-\hat{\delta}(q_j)^{-3}}<1/{2^2}.
$$
For general $k>2$, we may choose $\mu_k$ by induction such that
$$
|\zeta^k|^2 \sum_{\mu=\mu_k}^\infty \sum_{q_j\in E_{\nu^{k-1}_{2\mu-1}}} e^{-\hat{\delta}(q_j)^{-3}}<1/2^k
$$
where $\{\nu_{\mu}^{k-1}\}_{\mu= 1}^\infty$ is a renumeration of the sequence $\{\nu_{2\mu}^{k-2}\}_{\mu= 1}^\infty$.

  Now define a sequence $\{c^j\}_{j=1}^\infty$ of complex $m-$vectors by $c^j=\zeta^k$ for all $j$ with $q_j\in E_{\nu^{k-1}_{2l-1}}$,  $l\ge \mu_k$, and $c^j=0$ otherwise. Thus we have
  $
  \sum_{j=1}^\infty |c^j|^2e^{-\hat{\delta}(q_j)^{-3}}\le 1,
  $
  so that
  $$
 \sum_{j=1}^\infty |c^j|^2 e^{-a\,\hat{\delta}(q_j)^{-4}}<\infty
 $$
 for any $a>0$. Thus there exists a map $F\in {\mathcal O}(\Omega,{\mathbb C}^m)$ such that $F(q_{j})=c^j$ for each $j$, and
   $$
 |F(z)|\le C_1 e^{C_2 \hat{\delta}(z)^{-6}}, \ \ \ z\in \Omega,
$$
  in view of Proposition 2.2/(1). By the construction of the sequence $\{q_j\}_{j=1}^\infty$, we see that for each $k$ the set of cluster points of $\left\{q_j\in E_{\nu^{k-1}_{2l-1}}: l\ge \mu_k\right\}$ contains  $\partial \Omega\cup \{\infty\}$ ($\partial \Omega$ if $\Omega$ is bounded).  Thus the sequence $\{\zeta^k\}_{k=1}^\infty$ is contained in the cluster set of $F$ at any $p\in \partial \Omega\cup \{\infty\}$ ($\partial \Omega$ if $\Omega$ is bounded), so is the whole of ${\mathbb C}_\infty^m$.

  The second conclusion may be proved similarly by using Proposition 2.3 in place of Proposition 2.2.
\end{proof}

\begin{remark}
Let\/ ${\rm{CW}}(\Omega,{\mathbb C}^m)\subset \mathcal{O}(\Omega,{\mathbb C}^m)$ denote the set of all holomorphic maps whose sets of Casorati-Weierstrass points coincide with $\partial \Omega\cup \{\infty\}$ ($\partial \Omega$ if $\Omega$ is bounded). By a similar argument, invoking Proposition 2.5 in place of Proposition 2.2, one can show that ${\rm{CW}}(\Omega,{\mathbb C}^m)$ lies dense in $\mathcal{O}(\Omega,{\mathbb C}^m)$.
\end{remark}

As a consequence of Theorem 1.1 we immediately get the following

\begin{proposition}
An irreducible complex space admits a holomorphic map to ${\mathbb C}^m$ with dense image if and only if it admits a nonconstant holomorphic function.
\end{proposition}

\begin{proof} It suffices to verify the if part. Since every nonconstant holomorphic function on an irreducible complex space defines an open map to ${\mathbb C}$, we only need to construct a holomorphic map from an open set in $\mathbb C$ to $\mathbb C^m$ with dense image. But each open set in ${\mathbb C}$ is pseudoconvex, hence Theorem 1.1 applies.
\end{proof}

We also have the following analogous result in real-analytic category:

\begin{proposition} Each noncompact real-analytic manifold admits a real-analytic map to ${\mathbb R}^m$ with dense image.
\end{proposition}

\begin{proof} By virtue of  Grauert's theorem \cite{Grauert58}, each real-analytic manifold $M$ admits a Stein neighborhood $U$ in the total space of the tangent bundle $TM$ of $M$, with respect to the natural complexification. Suppose now $M$ is noncompact, then there is a discrete sequence $\{q_k\}_{k=1}^\infty$ of points in $M$. Let $\{\zeta^j\}_{j=1}^\infty$ be a dense sequence of points in ${\mathbb R}^m\subset {\mathbb C}^m$. By virtue of Proposition 2.5, one can construct a holomorphic map $F:U\rightarrow {\mathbb C}^m$ such that $F(q_{k_j})=\zeta^j$ for some subsequence $\{q_{k_j}\}$. The restriction of $F$ to $M$ gives the desired real-analytic map.
\end{proof}
\section{Universally dominated spaces}

By virtue of Proposition 3.1 we immediately get

\begin{proposition} An irreducible complex space is universally dominated if and only if it is dominated by some ${\mathbb C}^m$, and this holds if and only if it is dominated by ${\mathbb C}$.
\end{proposition}

Basing on this fact, we obtain the following

\begin{proposition} Suppose $M$ is universally dominated. Then we have
\begin{enumerate}
 \item  The Kobayashi pseudodistance $k_M$ of $M$ vanishes identically.

  \item $M$ is ultra-Liouville, i.e., any negative continuous psh function on $M$ is constant.

  \item If $M$ is a projective algebraic manifold, then the irregularity of $M$, i.e.,
  the dimension of the vector space of holomorphic $1-$forms on $M$, is no greater than the dimension of $M$.

  \item If $M$ is a domain in ${\mathbb C}^n$, then for any complex line $L$, $\pi_L(M)$ omits
  at most one point in $L$ where $\pi_L$ is the projection from ${\mathbb C}^n$ to $L$.
  \end{enumerate}
\end{proposition}

\begin{proof} (1) Take a dominant morphism $F:{\mathbb C}\rightarrow M$ (i.e., a holomorphic map with dense image). Given two points $z,w\in {\mathbb C}$, we have
$$
  k_{M}(F(z),F(w))\le k_{\mathbb C}(z,w)=0.
$$
Thus $k_M$ vanishes on a dense set of $M\times M$, so that $k_M$ has to vanish on $M\times M$ by continuity.

    (2) Suppose on the contrary that there exists a negative nonconstant continuous psh function $\psi$ on $M$. Let $F:{\mathbb C}\rightarrow M$ be a dominant morphism. Since $\psi$ is continuous, it is nonconstant on $F({\mathbb C})$. Thus $\psi\circ F$ would be a nonconstant negative subharmonic function on ${\mathbb C}$, which is absurd.

    (3) Suppose on the contrary that the irregularity of $M$ is greater than the dimension of $M$. By Bloch's theorem (cf. \cite{Ochiai77}, \cite{Noguchi77}, \cite{GreenGriffiths}, \cite{Kawamata80}), we known that every holomorphic map $F:{\mathbb C}\rightarrow M$ has its image in a closed proper subvariety of $M$. Thus $M$ is not universally dominated. Contradiction.

    (4) Since $\pi_L$ is an open map, $\pi_L(M)$ is an open set in $L$. If $L\backslash \pi_L(M)$ contains at least
    two points, then $\pi_L(M)$ is Kobayashi hyperbolic, so that $k_M$ does not vanish. By (1), $M$ could
    not be universally dominated. Contradiction.
\end{proof}

\begin{proposition}

\begin{enumerate}
\item  Let $M_1,\,M_2$ be two complex spaces. The product $M_1\times M_2$ is universally dominated if and only if both $M_1,\,M_2$ are universally dominated.

 \item  Let $M_1$ be a universally dominated space. If $M_2$ is dominated by $M_1$, then it is also
 universally dominated.
 \end{enumerate}
\end{proposition}

\begin{proof} (1) The if part: take two dominant morphisms $F_j:{\mathbb C}\rightarrow M_j$. It follows immediately that $(F_1,F_2):{\mathbb C}^2\rightarrow M_1\times M_2$ is dominant. The only if part: take a dominant morphism ${\mathbb C}\rightarrow M_1\times M_2$. By composing with the projections $\pi_j:M_1\times M_2\rightarrow M_j$, $j=1,2$, respectively, we get dominant morphisms ${\mathbb C}\rightarrow M_j$.

 (2) Take first a dominant morphism from ${\mathbb C}$ to $M_1$. By composing with a dominant morphism $M_1\rightarrow M_2$, we get a dominant morphism ${\mathbb C}\rightarrow M_2$.
\end{proof}

Now we give some examples of universally dominated spaces.

\begin{example}[4.1] The Riemann surfaces which are universally dominated are ${\mathbb C}$, ${\mathbb C}^\ast$, ${\mathbb P}^1$ and all tori. They are precisely all the Riemann surfaces which are Oka; the others are hyperbolic (cf. \cite{ForstnericBook}, Chapter 5).  Furthermore, any (singular) complex curve dominated by ${\mathbb C}$ or ${\mathbb P}^1$ is universally dominated, e.g., the rational normal curve in ${\mathbb P}^n$, which is defined to be the image of the holomorphic map ${\mathbb P}^1\rightarrow {\mathbb P}^n$ given by
$$
(z_0:z_1)\mapsto \left(z_0^n:z_0^{n-1}z_1:\cdots:z_0z_1^{n-1}:z_1^n\right).
$$
\end{example}

A complex manifold $Y$ is called an Oka manifold if every holomorphic map $F:K\rightarrow Y$ from a compact convex set $K\subset {\mathbb C}^n$ can be approximated uniformly on $K$ by entire maps ${\mathbb C}^n\rightarrow Y$.
For a list of examples of Oka manifolds, we refer to \cite{ForstnericBook} or \cite{ForstnericLarusson}. In particular, every complex Lie group is Oka.

\begin{example}[4.2] Elliptic $K3-$surfaces and Kummer surfaces (cf. \cite{BuzzardLu}).
\end{example}

\begin{example}[4.3] Fatou-Bieberbach domains, and unbounded domains ${\mathbb C}^n\backslash K$, where $K$ is a compact polynomially convex set in ${\mathbb C}^n$, $n\ge 2$ (cf. Corollary 2 in \cite{ForstnericRitter}, see also \cite{RosayRudin}).
\end{example}

\begin{remark} We claim that there are domains $\Omega_1\subset \Omega_2\subset {\mathbb C}^n$ with $n\ge 2$ such that
$\Omega_1$ is universally dominated while $\Omega_2$ is not. To see this, simply take $\Omega_1:={\mathbb C}^n\backslash {\mathbb B}^n$. Then it is universally dominated by
virtue of the previous example. Fix a point $p$ in the unit sphere and put $\Omega_2=\Omega_1\bigcup B(p,1)$. We claim
that $\Omega_2$ is not universally dominated. To see this, simply take a continuous psh peak function of $\Omega_2$ at the
 strongly pseudoconvex point $0$, so that $\Omega_2$ is not universally dominated by virtue of Proposition 4.2/(2).
\end{remark}

\begin{example}[4.4] Toric spaces, i.e., complex spaces with an open dense subset biholomorphic to $({\mathbb C}^\ast)^n$. One
warning: this definition is more general than the classical definition of toric varieties in algebraic geometry). Indeed, toric varieties are also Oka (cf. Theorem 2.17 in \cite{Forstneric13}). Classical examples of compact toric manifolds are the projective space ${\mathbb P}^n$, the Osgood space ${\mathbb C}^n_\infty$, and the Hirzebruch surfaces. Moreover, the Grassmannians are also toric manifolds. To see this, simply note that every Grassmannian ${\mathbb G}(k,n)$ contains ${\mathbb C}^{k(n-k)}$ as a Zariski open subset (see e.g., \cite{ChirkaAnalytic}, p. 320--321). Indeed,  Grassmanians are also complex homogeneous and hence Oka in view of the classical results of Grauert.
\end{example}

\begin{example}[4.5]  An important class of noncompact toric manifolds may be constructed as follow. Let $S$ be a closed subset in
${\mathbb C}^n$ so that there exists an automorphism $F$ of ${\mathbb C}^n$ such that $F(S)$ is contained in
complex coordinate hyperplanes. Clearly,  ${\mathbb C}^n\backslash S$ is biholomorphic to ${\mathbb C}^n\backslash F(S)$
 which contains $({\mathbb C}^\ast)^n$ as a dense subset, hence is a toric manifold. This applies in particular, to the
  complement of a tame set in the sense of Rosay-Rudin \cite{RosayRudin} in ${\mathbb C}^n$, $n\ge 2$.
  \end{example}

\begin{example}[4.6] Quotients of universally dominated manifolds by discrete groups of automorphisms acting properly discontinuously. This
includes all complex tori, the Iwasawa manifold and the Hopf manifolds.
\end{example}

\begin{example}[4.7] Calabi-Eckmann manifolds. In fact, we know from \cite{CalabiEckmann} that every Calabi-Eckmann manifold $M_{m,n}$ is a complex manifold homeomorphic to the Cartesian product ${\mathbb S}^{2m+1}\times {\mathbb S}^{2n+1}$ of two odd-dimensional spheres, and one can choose a cover of coordinate domains $V_{\alpha\beta}\,(\alpha=0,\cdots,m; \beta=0,1,\cdots,n)$ defined by
$$
V_{\alpha\beta}=\left\{(z,z')\in {\mathbb S}^{2m+1}\times{\mathbb S}^{2n+1}\subset {\mathbb C}^{m+1}\times {\mathbb C}^{n+1}: z_\alpha z_\beta'\neq 0\right\},
$$
which is homeomorphic to ${\mathbb C}^{m+n}\times {\mathbb T}^1$ where ${\mathbb T}^1$ is a $1-$dimensional complex torus, so that the complex structure of ${\mathbb C}^{m+n}\times {\mathbb T}^1$ gives local coordinates of $V_{\alpha\beta}$. It follows immediately that $M_{m,n}$ is universally dominated.
\end{example}

\begin{example}[4.8] Let $M$ be a universally dominated manifold and $f$ a holomorphic function on $M$. Then the complement of the graph of $f$ in $M\times {\mathbb C}$ is universally dominated. Indeed,
it is biholomorphic to $M\times{\mathbb C}^\ast$ via the automorphism $({\rm id}_M,{\rm id}_{\mathbb C}-f)$ of $M\times {\mathbb C}$. It is unknown whether the condition can be weaken to that $f$ is only meromorphic. The special case when $M={\mathbb C}$ has been verified by Buzzard-Lu (cf. \cite{BuzzardLu}, Theorem 5.2). There are new results in this direction (cf.  Section 4 in \cite{Hanysz2012}).
\end{example}

\begin{example}[4.9]  The complement $M$ of $k\le n+1$ distinct hyperplanes in general position in ${\mathbb P}^n$ is universally dominated. To see this, simply take a complex chart $({\mathbb C^n;z})$ in ${\mathbb P}^n$ so that the restriction of $M$ to this complex chart is $({\mathbb C}^\ast)^{k-1}\times {\mathbb C}^{n-k+1}$, which is dense in ${\mathbb P}^n$.  On the other side,  the image of any holomorphic map from ${\mathbb C}$ to the complement of $n+2$ distinct hyperplanes in ${\mathbb P}^n$ lies in a hyperplane (cf. \cite{Green72}),
thus the latter is not universally dominated.
\end{example}

In general, it is very difficult to determine whether the complement of a given divisor $S$ in  a non-hyperbolic
manifold is universally dominated (e.g., Kobayashi's conjecture). Below  we present a useful method of constructing
divisors with universally dominated complements in ${\mathbb C}^n$ or ${\mathbb P}^n$. Put
$$
{\mathcal F}_n:=\left\{f\in {\mathcal O}({\mathbb C}^n): {\mathbb C}^n\backslash Z_f \ {\rm is\ universally\ dominated\,}\right\}
$$
where $Z_f=\{f=0\}$.
Clearly, $z_1^{\alpha_1}\cdots z_n^{\alpha_n}\in {\mathcal F}_n$ for any nonnegative integers $\alpha_1,\cdots,\alpha_n$. Put
$$
S=\left\{(z_1,\cdots,z_n)\in {\mathbb C}^n: f(z_1,\cdots,z_{n-1})z_n+g(z_1,\cdots,z_{n-1})=0\right\}
$$
where $f\in {\mathcal F}_{n-1}$ and $g\in {\mathcal O}({\mathbb C}^{n-1})$. We claim that ${\mathbb C}^n\backslash S$ is universally dominated. To see this, simply view  $\left(\left({\mathbb C}^{n-1}\backslash Z_f\right)\times \mathbb C\right)\cap S$ as the graph of the holomorphic function
$$
h:=-g/f
$$
on ${\mathbb C}^{n-1}\backslash Z_f$ so that its complement in $\left({\mathbb C}^{n-1}\backslash Z_f\right)\times {\mathbb C}$ is universally dominated by Example 4.8, so is ${\mathbb C}^n\backslash S$.

It follows immediately that ${\mathbb P}^{n}\backslash \hat{S}$ is universally dominated if $\hat{S}$ is defined by
$$
 z_1^{\alpha_1}\cdots z_{n}^{\alpha_{n}}z_0+\sum_{\gamma_1+\cdots+\gamma_n=\alpha_1+\cdots+\alpha_n+1} c_{\gamma_1\cdots\gamma_n}z_1^{\gamma_1}\cdots z_n^{\gamma_n}=0
$$
where $\alpha_1,\cdots,\alpha_n$ are nonnegative integers. As a consequence, we get

\begin{proposition}
\begin{enumerate}
\item  The complement of any smooth quadric hypersurface in ${\mathbb P}^n$ is universally dominated.

\item  The complement of the universal hypersurface of degree $d$ in ${\mathbb P}^m\times{\mathbb P}^n$ with
 $m=\left(
                                                                                                                     \begin{array}{c}
                                                                                                                       n+d \\
                                                                                                                       n \\
                                                                                                                     \end{array}
                                                                                                                   \right)
                                                                                                                   -1
 $
 is universally dominated.
 \end{enumerate}
 \end{proposition}

\begin{proof} (1)  A quadric hypersurface in ${\mathbb P}^n$ is defined by a homogeneous polynomial of degree 2.
After a change of coordinates, we may assume that hypersurface is
    $$
Q_k:=\left\{(z_0:z_1:\cdots:z_n): z_0^2+z_1^2+\cdots+z_k^2=0\,\right\}
 $$
 where $k$ is the rank of $Q_k$. Clearly, $Q_k$ is smooth if and only if $k=n$. We claim that ${\mathbb P}^n\backslash Q_n$ is universally dominated.  To see this, take first a change of coordinates as follows
  $$
 z_0=\zeta_0+i\zeta_1,\,z_1=\zeta_1+i\zeta_2,\ \cdots,\ z_n=\zeta_n+i\zeta_0,
 $$
 so that the equation becomes
 $$
\zeta_0\zeta_1+\zeta_1\zeta_2+\cdots+\zeta_n\zeta_0=0.
$$
Next choose a new coordinate system:
$$
t_0=\zeta_0,\ \cdots,\ t_{n-1}=\zeta_{n-1},\ t_n=\zeta_1+\zeta_n,
$$
so that
$$
Q_n=\left\{(t_0:t_1:\cdots:t_n):t_0 t_n +{\rm polynomial\ of\ }(t_1,\cdots,t_n)=0\right\}.
$$
Thus ${\mathbb P}^n\backslash Q_n$ is universally dominated.

(2) Recall that the universal hypersurface ${\mathcal S}_{n,d}$ of degree $d$ in ${\mathbb P}^n$ is defined by
$$
\sum_{\alpha_0+\cdots+\alpha_n=d} t_{\alpha_0\cdots\alpha_n} z_0^{\alpha_0}\cdots z_n^{\alpha_n}=0
$$
where $(z_0:z_1:\cdots:z_n)$ is the homogeneous coordinate of ${\mathbb P}^n$ and
$(t_{\alpha_0\cdots\alpha_n})_{\alpha_0+\cdots+\alpha_n=d}$ is the homogeneous coordinate of ${\mathbb P}^m$.
 It is easy to see that $({\mathbb P}^m\times {\mathbb P}^n)\backslash {\mathcal S}_{n,d}$ is universally dominated.
\end{proof}

The central property of universally dominated spaces is the following

\begin{theorem} {\it Let $M_1$ be a universally dominated complex space and $M_2$ a complex space.
If there exists a surjective meromorphic map\/ $\Phi:M_1\rightarrow M_2$, then $M_2$ is also universally dominated.}
\end{theorem}

\begin{proof} The argument is inspired by Kobayashi \cite{KobayashiBook98}, Lemma 3.5.29. Take first
 a dominant morphism $F:{\mathbb C}\rightarrow M_1$. Let $S$ be the singular set of $\Phi$. Since $\overline{F({\mathbb C})}=M_1$ and $\Phi$ is surjective, $F^{-1}(S)\neq {\mathbb C}$, i.e., $F^{-1}(S)$ is nowhere dense in ${\mathbb C}$.
  Thus $\Phi\circ F:{\mathbb C}\rightarrow M_2$ is a meromorphic map (in the sense of Remmert) with singularity set $F^{-1}(S)$.
   It is known from Remmert \cite{Remmert57} (see also \cite{KobayashiBook98}, p. 38) that the singular set of a meromorphic map is a closed complex subspace of codimension $\ge 2$.  Since  ${\rm dim\,}{\mathbb C}=1$, we conclude that the singular set of $\Phi\circ F$ has to be empty,
    i.e., $\Phi\circ F$ is actually holomorphic. Since there is an open dense subset $U\subset M_1$ so that
    $\Phi:U\rightarrow M_2$ is a dominant morphism, we see that $\Phi\circ F:{\mathbb C}\rightarrow M_2$ is
    also a dominant morphism.
\end{proof}

An immediate consequence is

\begin{corollary}
 Universal dominability is a bimeromorphic invariant.
\end{corollary}

\begin{example}[4.10] Unirational varieties and Kummer manifolds are universally dominated.
 Indeed, if $M$ is unirational, then there exists a surjective rational map
 $F:{\mathbb P}^n \rightarrow M$. Let $\pi:\widetilde{M}\rightarrow M$ be
 a desingularization of $M$. Clearly, $\pi^{-1}\circ F:{\mathbb P}^n\rightarrow \widetilde{M}$
 is a surjective rational map. Thanks to Theorem 4.5, $\widetilde{M}$ is universally dominated,
 so is $M$ (see also \cite{Winkelmann05} for a different approach). For a Kummer manifold $M$,
 there exist an abelian variety $A$ and a finite group $G$ of holomorphic automorphisms of $A$ such that
 $M$ is bimeromorphically equivalent to the quotient variety $A/G$, thus it has to be universally dominated, in view of Corollary 4.6. Note that Kummer surfaces are also strongly dominable by ${\mathbb C}^2$ (cf. Corollary 3 and \S 3 in \cite{ForstnericLarusson}).
\end{example}

It is interesting to ask whether universal dominability is stable under small deformations of complex structures. The answer is no. It is known form Forstneri$\check{\rm c}$-L\'arusson \cite{ForstnericLarusson} that there exists a complex analytic
family of\/ {\it compact}\/ complex surfaces such that the central fibre is an Inoue-Hirzebruch surface
 which is not universally dominated since its universal covering possesses a nonconstant negative psh function,
 continuous outside of a curve, whereas all other fibres are minimal Enoki surfaces which are Oka, so they are
  universally dominated. Such a family is actually due to Dloussky. We thank Finnur L\'arusson for pointing out this fact.

 We have learnt from Proposition 4.2/(1) that universal dominability implies the Kobayashi pseudodistance $k_M\equiv 0$. The converse fails, however:

\begin{proposition} {\it There is a complex space $M$ which is not universally dominated but $k_M\equiv 0$.}
\end{proposition}

\begin{proof} We start with a hyperplane ${\mathbb P}^{n-1}\subset {\mathbb P}^n$ and a point $p\in {\mathbb P}^n\backslash {\mathbb P}^{n-1}$. Fix a nonsingular curve ${\mathcal C}\subset {\mathbb P}^{n-1}$ of genus $\ge 2$. The cone over ${\mathcal C}$ with vertex $p$ is defined as
  $$
  {\rm cone\,}({\mathcal C},p)=\bigcup_{q\in {\mathcal C}}\overline{qp},
  $$
  i.e., the union of the complex lines jointing $p$ to points of ${\mathcal C}$. Since ${\rm cone\,}({\mathcal C},p)$ is
  a union of ${\mathbb P}^1$'s intersects at $p$, its Kobayashi pseudodistance vanishes identically
  (compare \cite{KobayashiBook98}, Example 3.2.21). On the other hand, since the projection
  $\pi_p:{\rm cone\,}({\mathcal C},p)\rightarrow {\mathcal C}$ is a surjective rational map,
  ${\rm cone\,}({\mathcal C},p)$ cannot be universally dominated, in view of Theorem 4.5.
\end{proof}

We conclude this section by proposing a few open problems.

\begin{problem} Suppose $M$ is a universally dominated manifold and $\widetilde{M}$ is an unramified holomorphic covering of $M$. Is $\widetilde{M}$ universally dominated?

\end{problem}

\begin{problem} Let $M$ be a holomorphic fiber bundle. Suppose both the base and the fiber are universally dominated. Is $M$ also universally dominated?

\end{problem}

Note that the Oka property has the following property: If $\pi:X\rightarrow Y$ is a holomorphic fiber bundle with an Oka fiber, then $X$ is Oka iff $Y$ is Oka (cf. \cite{ForstnericBook}, Chapter 5). This holds in particular for unbranched coverings.

 \begin{proposition} Let $M$ be a universally dominated manifold and $E$ a holomorphic principle $G-$bundle over
 $M$ with $G$ being a connected complex Lie group (every complex Lie group is Oka and hence universally dominated). Then $E$ is also universally dominated.
\end{proposition}

\begin{proof} Take a dominant morphism $F:{\mathbb C}\rightarrow M$. The pullback bundle,
$F^\ast (E)$ over ${\mathbb C}$ is holomorphically trivial, thanks to Grauert's Oka-principle (cf. \cite{GrauertOka}).
It follows that $F^\ast (E)$ is universally dominated, so is $E$.
\end{proof}

 In particular, the Stiefel manifold ${\rm St}(k,n)$, i.e., the set of all $k-$tuples of linearly independent vectors in ${\mathbb C}^n$, is universally dominated. To see this, simply view ${\rm St}(k,n)$ as a principle fiber bundle over the Grassmannian ${\mathbb G}(k,n)$, with fiber ${\rm GL}(k,{\mathbb C})$. Indeed, every Stiefel manifold is also homogeneous and Oka.

\begin{problem} Is every projective algebraic manifold of general type\/ {\it not} universally dominated?
\end{problem}

 Of course, a resolution of the celebrated Green-Griffiths conjecture (cf. \cite{GreenGriffiths}) would give a positive answer to this problem.

\begin{problem} Suppose $M$ is a universally dominated manifold of dimension $n$ and $S$ is a closed subset in $M$ so that $M\backslash S$ is connected. Under which condition is $M\backslash S$ universally dominated?
\end{problem}

For instance, we do not know whether the complement $M$ of the totally real plane
$$
\{(z_1,z_2):{\rm Re\,}z_1={\rm Re\,}z_2=0\}
$$
 in ${\mathbb C}^2$ is universally dominated. It is known that $M$ contains a Fatou-Bieberbach domain (cf. \cite{RosayRudin}, Example 9.6).

\section{Picard points}

We recall first the following classical Lindel\"of principle:

\begin{proposition} [cf. \cite{Tsuji59}, p. 308]  Let $D$ be a domain in ${\mathbb C}$ which is bounded by a Jordan curve $C$ such that $0\in C$. The part of $C$ which lies in a neighborhood of $0$ is decomposed by $0$ into two parts $C_1,C_2$. If $f$ is a holomorphic function on $D$ satisfying $\lim_{z\rightarrow 0}f(z)=a$ when $z\rightarrow 0$ on $C_1$ and\/ $\lim_{z\rightarrow 0}f(z)=b$ when $z\rightarrow 0$ on $C_2$ and $a\neq b$, then $f$ takes any value infinitely often in $D$, with one possible exception.
\end{proposition}

\begin{corollary} Let $\Omega$ be a domain in ${\mathbb C}$ and $p\in\partial \Omega$. Suppose $\Omega$ contains a cone $\Lambda_p$ with vertex at $p$.  Then there is a holomorphic function $f$ on ${\mathbb C}\backslash \{p\}$ such that $p$ is a Picard point of $f$ on $\Omega$ and $f'(z)\neq 0$ for each $z\in {\mathbb C}\backslash \{p\}$.
\end{corollary}

\begin{proof} Put $f_1(z)=e^{1/z^2}$. Clearly, $f_1\in \mathcal{O}({\mathbb C}^\ast)$. Write $z=x+iy$. Since
$$
f_1(z)=e^{\frac{x^2-y^2}{(x^2+y^2)^2}}e^{-\frac{2ixy}{(x^2+y^2)^2}},
$$
we see that for each $0<\varepsilon<1$, $\lim_{z\rightarrow 0}f_1(z)=\infty$ when $z\rightarrow 0$ on the ray $l_1:y=(1-\varepsilon)x$, $x>0$, and $\lim_{z\rightarrow 0}f_1(z)=0$ when $z\rightarrow 0$ on the ray $l_2:y=(1+\varepsilon)x$, $x>0$. By virtue of the previous proposition, we conclude that $0$ is a Picard point for $f_1$ on the following (infinite) cone
$$
V_\varepsilon=\left\{(x,y)\in {\mathbb R}^2:(1-\varepsilon)x<y<(1+\varepsilon)x,\ x>0\right\}.
$$
For sufficiently small $\varepsilon$, we have a complex affine map $\Phi:V_\varepsilon\cap B(0,\varepsilon) \rightarrow \Lambda_p$ obtained by a composition of translation and rotation such that $\Phi(0)=p$. The desired function may be chosen as $f=f_1\circ \Phi^{-1}$.
\end{proof}

\begin{remark} The ray $y=x, x>0$ is a Julia ray of $f$. Julia proved that each $f\in {\mathcal O}({\mathbb C}^\ast)$ with $0$ as essential singularity has at least one Julia ray.
\end{remark}

Next we recall two types of Lindel\"of principles of several complex variables as follows.

\begin{definition} Let $M,N$ be complex manifolds, and ${\mathcal O}(M,N)$ be the set of holomorphic maps from $M$ to $N$. A map $F\in {\mathcal O}(M,N)$ is called\/ {\it normal}\/ if the family $\{F\circ H:H\in {\mathcal O}({\mathbb D},M)\}$ forms a normal family in the sense of Wu \cite{Wu_Normal}.
\end{definition}

We remark that for a bounded domain $\Omega\subset {\mathbb C}^n$, a function $f\in {\mathcal O}(\Omega)$ is normal if it omits (at least) two values (cf. \cite{CimaKrantz}).

\begin{definition} Let $\Omega\subset {\mathbb C}^n$ be a bounded domain with $C^2-$boundary. Let $f\in {\mathcal O}(\Omega)$ and $p\in \partial\Omega$. We say that $f$ has non-tangential limit $c$ at $p$ if
$$
\lim_{\Gamma_\alpha(p)\ni z\rightarrow p} f(z)=c
$$
for any $\alpha>0$, where
$$
\Gamma_\alpha(p)  =  \left\{z\in \Omega:|z-p|<(1+\alpha)\delta_\Omega(z)\right\}.
$$
A curve $\gamma:[0,1)\rightarrow \Omega$ which terminates at $p$ is said to be non-tangential if it is contained in some $\Gamma_\alpha(p)$.
\end{definition}

We have the following Lindel\"of principle due to Cima and Krantz:

\begin{proposition}[cf. \cite{CimaKrantz}, Lemma 3.1 and the subsequent remark]  Let $\Omega\subset {\mathbb C}^n$ be a bounded domain with $C^2-$boundary and $p\in \partial \Omega$. Suppose $f\in {\mathcal O}(\Omega)$ is a normal function, and $f$ has limit $c$ along some non-tangential curve $\gamma$ terminating at $p$. Then $f$ has non-tangential limit $c$ at $p$.
\end{proposition}

Generalizing the classical results of Bagemihl-Seidel \cite{BagemihlSeidel} from one complex variable, K. T. Hahn proved the following

\begin{proposition}[cf. \cite{HahnNormal}, Theorem 4 and the subsequent remark] Let $\Omega\subset {\mathbb C}^n$ be a bounded domain with $C^2-$boundary at $p\in \partial \Omega$. Let $\{p_j\}_{j=1}^\infty$ be a sequence of points in $\Omega$ which tends to $p$ such that
\begin{enumerate}
\item $\lim_{j\rightarrow \infty} k_\Omega(p_j,p_{j+1})=0$,

\item $\lim_{j\rightarrow \infty} \frac{{\rm dist\,}(p_j,{\mathbb C}\nu_p)^2}{{\rm dist\,}(p_j,{\mathbb C}T_p)}=0$.
\end{enumerate}
If $f\in {\mathcal O}(\Omega)$ is a normal function such that $f(p_j)\rightarrow c$ as $j\rightarrow \infty$, then $f$ has non-tangential limit $c$ at $p$. Here ${\mathbb C}T_p$ and ${\mathbb C}\nu_p$ are the complex tangent space and the complex normal space at $p$, respectively, and $k_\Omega$ is the Kobayashi distance of $\Omega$.
\end{proposition}

\begin{remark} Actually, Hahn's theorem is much more general. He considered normal maps from $\Omega$ to a relatively compact open set $X$ in a hermitian manifold $N$ with hermitian distance $d_N$. Note that ${\mathbb C}$ may be regarded as a relatively compact domain in the Riemann sphere with Fubini-Study distance.
\end{remark}

\begin{lemma}
Let ${\mathbb B}$ denote the unit ball in ${\mathbb C}^n$. Let $\{p_j=(t_j,0,\cdots,0)\}_{j=1}^\infty$ be a sequence of points in ${\mathbb B}$ defined inductively by
\begin{equation}
t_1=1-\frac{1}{e^2},\ \ \ t_{j+1}=t_j+(1-t_j)|\log (1-t_j)|^{-1}, \ \ \ j\ge 1.
\end{equation}
Then $p_j\rightarrow (1,0,\cdots,0)$ and\/ $k_{\mathbb B}(p_j,p_{j+1})\rightarrow 0$ as $j\rightarrow \infty$.
\end{lemma}

\begin{proof}
Note that $\{t_j\}_{j=1}^\infty$ is an increasing sequence of positive numbers, thus $t_j\rightarrow a$ for suitable $t_1\le a\le 1$. Let $j\rightarrow \infty$ in (5.1), we immediately get $a=1$, i.e., $p_j\rightarrow (1,0,\cdots,0)$.

 Since
$$
\frac{|t_{j+1}-t_j|}{|1-t_jt_{j+1}|}=\frac{|\log(1-t_j)|^{-1}}{1+t_j-t_j|\log(1-t_j)|^{-1}}\rightarrow 0,
$$
we have
$$
k_{\mathbb B}(p_j,p_{j+1})\le k_{\mathbb D}(t_j,t_{j+1})=\log \frac{1+\frac{|t_{j+1}-t_j|}{|1-t_jt_{j+1}|}}{1-\frac{|t_{j+1}-t_j|}{|1-t_jt_{j+1}|}}\rightarrow 0.
$$
\end{proof}

\begin{lemma} Let $\Omega$ be a domain in ${\mathbb C}^n$ and $f\in \mathcal{O}(\Omega)$.  If there is a dense set $\{p_j\}_{j=1}^\infty\subset \partial \Omega$ of Picard points for $f$, then each point in $\partial\Omega\cup \{\infty\}$  ($\partial \Omega$ if $\Omega$ is bounded) is a Picard point of $f$. Indeed, the set of Picards points is always closed.
\end{lemma}

\begin{proof} The argument is elementary (cf. \cite{CollingwoodLohwater}). Suppose first $\Omega$ is bounded. Let $p\in \partial \Omega$ be given. Suppose $p$ is not a Picard point for $f$. Then there would be a number $r>0$ such that $f(B(p,r)\cap \Omega)$ omits at least two complex numbers $a,b$. Since $\{p_j\}_{j=1}^\infty$ is dense in $\partial \Omega$, we have at least one point $p_{j_0}\in B(p,r/2)$. Thus $f(B(p_{j_0},r/2)\cap \Omega)$ also omits $a,b$, so that $p_{j_0}$ is not a Picard point of $f$, a contradiction. The case when $\Omega$ is unbounded is similar.
\end{proof}

\begin{proof}[Proof of Theorem 1.2.] We show first that there are a dense sequence $\{p_j\}_{j=1}^\infty$ of points in $\partial \Omega$ and a mutually disjoint sequence of balls $\{B_j\}_{j=1}^\infty$ in $\Omega$ such that $\partial B_j $ touches $\partial \Omega$ only at $p_j$ for each $j$. To see this, take first a dense sequence $\{q_j\}_{j=1}^\infty$ of points in $\partial \Omega$. If $q_1$ is an isolated point of $\partial \Omega$, then we take $p_1=q_1$ and a small ball $B_1$ in $\Omega$ such that $\partial B_1$ touches $\partial \Omega$ only at $p_1$. Otherwise, we may choose $p_1^\ast\in B(q_1,1/2)\cap\Omega$ and $p_1\in \partial \Omega$ such that $|p_1^\ast-p_1|=\delta_\Omega(p_1^\ast)$. It suffices to take
$$
B_1=B((p_1^\ast+p_1)/2, \ \delta_\Omega(p_1^\ast)/2).
$$
Suppose we have chosen $p_j$ and $B_j$ for $1\leq j\leq k-1$. If $q_k$ is an isolated point of $\partial \Omega$, then we take $p_k=q_k$ and a small ball $B_k$ in $\Omega\backslash\bigcup_{j=1}^{k-1}\overline{B}_j$ such that $\partial B_k$ touches $\partial \Omega\cup \bigcup_{j=1}^{k-1}\partial B_j$ only at $p_k$. Otherwise, we may choose $p_k^\ast\in B(q_k,{1}/{2^k})\cap\left(\Omega\backslash\bigcup_{j=1}^{k-1}\overline{B}_j\right)$ and $p_k\in \partial \Omega$ such that
 $$
|p_k^\ast-p_k|=\delta_\Omega(p_k^\ast)<{\rm dist}\left(p_k^\ast, \bigcup_{j=1}^{k-1}\overline{B}_j\right).
$$
 It suffices to take
 $$
 B_k=B\left((p_k^\ast+p_k)/2, \delta_\Omega(p_k^\ast)/2\right).
 $$

By virtue of Lemma 5.5, we may choose in each $B_j$ a sequence $\{z^{j\mu}\}_{\mu=1}^\infty$ of points on the radius terminating at $p_j$
such that $\lim_{\mu\rightarrow \infty} z^{j\mu}=p_j$,
$\lim_{\mu\rightarrow \infty} k_{B_j}\left(z^{j\mu},z^{j\mu+1}\right)=0,
$
and
$$
B\left(z^{j\mu},\frac12\hat{\delta}(z^{j\mu})|\log \hat{\delta}(z^{j\mu})|^{-1}\right)\subset B_j
$$
 are mutually disjoint for all $j,\mu$.
By virtue of Proposition 2.2/(2) and Lemma 2.4, we find for each $\alpha>4$ a function $f\in {\mathcal O}(\Omega)$ such that $f(z^{j(2\mu-1)})=0$ and $f(z^{j(2\mu)})=1$ for any $j,\mu\ge 1$, and
 $$
 |f(z)|\le C_1 e^{C_2 \hat{\delta}(z)^{-\alpha-2}}, \ \ \ z\in \Omega.
$$
 We claim that each $p^j$ is a Picard point of $f$. Suppose on the contrary that $p_j$ is not a Picard point, then there would be a neighborhood $U_j$ of $p_j$ such that $f|_{\Omega\cap U_j}$ omits at least two values, in particular, it is a normal function on $B_j\cap U_j$. Applying Proposition 5.4 to the sequence $\{z^{j(2\mu-1)}\}_{\mu=1}^\infty$, we conclude that $f|_{B_j\cap U_j}$ has non-tangential limit $0$ at $p_j$, which is absurd. Combining with Lemma 5.6, we conclude the proof of the first conclusion.   The second conclusion may be proved similarly by using Proposition 2.3 in place of Proposition 2.2.
\end{proof}

To prove Proposition 1.3, we need a general extension theorem of Ohsawa as follows.

Let $\Omega\subset {\mathbb C}^n$ be a pseudoconvex domain and  let $S$ be a closed complex submanifold of $\Omega$ such that each component of $S$ is a domain in some complex affine subspace of ${\mathbb C}^n$. We denote by $\#(S)$ the set of all\/ {\it negative}\/  $%
\Psi \in PSH(\Omega)$ satisfying the following conditions:

(i) $S\subset \Psi ^{-1}(-\infty )$;

(ii) If $S$ is $k-$dimensional around a point $x$, there exists a local coordinates $(z_1,\cdots,z_n)$ on a neighborhood $U$ of $x$ such that $z_{k+1}=\cdots=z_n=0$ on $S\cap U$ and
\[
\sup_{U\backslash S}\left| \Psi (z)-(n-k)\log \sum_{j=k+1}^n |z_j|^2
\right| <\infty.
\]

Let $dV$ denote the Lebesgue measure in ${\mathbb C}^n$. For each $\Psi \in \#(S)$, one can define a positive measure $%
dV[\Psi ]$ on $S$ as the minimum element of the
partially ordered set of positive measure $d\mu $ satisfying
\[
\int_{S}fd\mu \geq \lim \sup_{t\rightarrow +\infty }\frac{2(n-k)}{%
\sigma _{2n-2k-1}}\int_{\Psi^{-1}((-t-1,-t))}fe^{-\Psi }dV
\]
for any nonnegative continuous function $f$ which is compactly supported in $%
\Omega$. Here $\sigma _m$ denotes the volume of the unit sphere in $\mathbb{R}%
^{m+1}$.

\begin{theorem} [cf. \cite{Ohsawa01}]  Let $\Omega$ be a pseudoconvex domain in $%
\mathbb{C}^n$ and let $S$ be as above. Let $\varphi$ be a psh function on $\Omega$ and $\Psi \in
\#(S)$. Then for any holomorphic function $f$ on $S$ satisfying $%
\int_S |f|^2 e^{-\varphi} dV[\Psi ]<+\infty $, there exists a
holomorphic function $F$ on $\Omega$ such that $F|_{S}=f$
and
\[
\int_\Omega |F|^2e^{-\varphi}dV\leq 2^8\pi \int_{S}|f|^2 e^{-\varphi}dV[\Psi
].
\]
\end{theorem}

\begin{proof}[Proof of Proposition 1.3] Let $\Omega_2$ denote the intersection of $\Omega$ with a $2-$dimensional complex affine subspace intersecting $\partial \Omega$ transversally at $p$ (note that $\Omega$ admits an inner ball at $p$). By virtue of the Ohsawa-Takegoshi extension theorem \cite{OhsawaTakegoshi87}, each function in $A^2_\alpha(\Omega_2)$ with $\alpha\ge 0$ extends to a function in $A^2_\alpha(\Omega)$. Thus it suffices to consider the case when $n=2$. Without loss of generality, we may assume that $p=0$ and the inner ball at $0$ is $\{(z_1,z_2)\in {\mathbb C}^2:|z_1-1|^2+|z_2|^2<1\}$. Put
 $$
 L_1=\{(z_1,z_2)\in {\mathbb C}^2:z_2=0\},\ \ \  L_2=\{(z_1,z_2)\in {\mathbb C}^2:z_2=z_1\},
 $$
and $S:=S_1\cup S_2$, where $S_k=\Omega\cap L_k$ for $k=1,2$. Let $f$ be a holomorphic function on $S$ given by $f=1$ on $S_1$ and $f=0$ on $S_2$. Let $d$ denote the diameter of $\Omega$. With respect to the functions
$$
\varphi(z):=\alpha \log 1/\delta_\Omega(z),\ \ \ \ \ \Psi(z):=\log |z_2|^2+\log |z_2-z_1|^2-2\log (2 d^2)\in \#(S),
$$
we have the estimate
$$
\int_S |f|^2 e^{-\varphi} dV[\Psi]=\int_{S_1} e^{-\varphi} dV[\Psi]=4 d^4\int_{S_1} |z_1|^{-2}\delta_\Omega^{\alpha}\le 4d^4 \int_{\{z_1:0<|z_1|<d\}} |z_1|^{-2+\alpha}<\infty.
$$
 Thus by the previous extension theorem, there is a holomorphic extension $F$ of $f$ to $\Omega$ such that
$$
\int_\Omega |F|^2e^{-\varphi}dV<\infty,
$$
i.e., $F\in A^2_\alpha(\Omega)$. Now $F=1$ along the real line $l_1:=\{(x,0)\in {\mathbb C}^2:x\in {\mathbb R}\}$ and $F=0$ along the real line $\{(x,x)\in {\mathbb C}^2:x\in {\mathbb R}\}$, we conclude that $0$ is a Picard point of $F$, in view of Proposition 5.3.

For the second assertion, we put $\varphi=0$ and
$$
\Psi(z):=-\frac4{\beta_1}\log(-\psi)+\log |z_2|^2+\log |z_2-z_1|^2-\frac4{\beta_1}\log C_1-\log 4.
$$
Since $-\psi(z)\ge C_1|z|^{\beta_1}$, we conclude that $\Psi\in \#(S)$. For the function $f$ defined as above, we have
$$
\int_S |f|^2 dV[\Psi]=\int_{S_1} dV[\Psi]=C\int_{S_1} |z_1|^{-2}(-\psi)^{4/\beta_1}\le C \int_{\{z_1:0<|z_1|<d\}} |z_1|^{-2+4\beta_2/\beta_1}<\infty
$$
since $-\psi(z)\le C_2 |z|^{\beta_2}$. Applying the extension theorem again, we see that $f$ admits a holomorphic extension $F\in A^2(\Omega)$. By a similar argument as above, we conclude that $0$ is a Picard point of $F$.
\end{proof}

\section{Holomorphic maps onto ${\mathbb C}^n$}

We begin with the following

\begin{proposition} An irreducible complex space admits a holomorphic map onto ${\mathbb C}$ if and only if it admits a nonconstant holomorphic function.
\end{proposition}

\begin{proof} It suffices to verify the if part. Let $h$ be a nonconstant holomorphic function on an irreducible complex space $M$. Then the image $h(M)$ is an open set in ${\mathbb C}$. Thanks to Corollary 5.2, there is a surjective holomorphic map $g:h(M)\rightarrow {\mathbb C}^\ast$. On the other side, the holomorphic function $f(z)=z+\frac1z$ maps ${\mathbb C}^\ast$ onto ${\mathbb C}$. Clearly, $f\circ g\circ h$ defines a surjective holomorphic map from $M$ to ${\mathbb C}$.
\end{proof}

In order to study locally biholomorphic maps, we need the following

\begin{lemma} There is a surjective holomorphic map $f:{\mathbb C}^\ast \rightarrow {\mathbb C}$ which is locally biholomorphic.
\end{lemma}

\begin{proof} Following E. Calabi (see \cite{StyerMinda74}, p. 640, or \cite{Zalcman74}, p. 135), the following non-polynomial entire function
$$
 f(z)=\int_0^z e^{-w^2} \/ dw.
$$
maps $\mathbb C$ locally biholomorphically onto $\mathbb C$.  Since $\infty$ is an essential singularity of $f$, there exists $a\in \mathbb C$ whose preimage contains at least two points, say $0,1$ for the sake of simplicity. It follows that $f|_{{\mathbb C}^\ast}$ is a desired holomorphic map.
\end{proof}

\begin{proposition} Every noncompact Riemann surface admits a holomorphic map onto $\mathbb C$ which is locally biholomorphic.
\end{proposition}

\begin{proof} Let $M$ be a noncompact Riemann surface. By virtue of Gunning and Narasimhan's theorem (cf. \cite{GunningNarasimhan67}), there is a function $h\in {\mathcal O}(M)$ without critical points. Its image $h(M)$ is a domain in ${\mathbb C}$. We only need to consider the case $h(M)\neq {\mathbb C}$. Since $h(M)$ has at least one boundary point which admits an inner ball, we infer from Corollary 5.2 that there is a surjective holomorphic map $f:h(M)\rightarrow {\mathbb C}^\ast$ without critical points. Thus $f\circ h$ defines a holomorphic map from $M$ onto ${\mathbb C}^\ast$, which is locally biholomorphic. The assertion follows immediately from the previous lemma. \end{proof}

Proposition 6.3 extends (with the same proof) to any Stein manifold $M$, in the form that there is a holomorphic submersion $f:M\rightarrow {\mathbb C}$ onto ${\mathbb C}$ (cf. \cite{ForstnericActa}).

\begin{corollary} Every Riemann surface admits a holomorphic map onto ${\mathbb P}^1$.
\end{corollary}

\begin{proof} It is well-known that each compact Riemann surface admits a holomorphic map onto ${\mathbb P}^1$. Thus we only need to consider noncompact Riemann surfaces. The desired holomorphic map may be obtained by  composing a holomorphic map onto ${\mathbb C}$, the universal covering map from ${\mathbb C}$ to a torus, and a surjective holomorphic map from this torus to ${\mathbb P}^1$.
\end{proof}

\begin{proposition} Let $\Omega$ be a domain in ${\mathbb C}^n$. Suppose there is a point $p\in \partial \Omega$ such that $\Omega$ contains a cone $\Lambda_p$ with vertex at $p$, and there is a supporting complex hyperplane of $\Omega$ at $p$. Then there is a holomorphic map from $\Omega$ onto ${\mathbb C}^n$ which is locally biholomorphic.
\end{proposition}

\begin{proof} After a change of (global) coordinate by a complex affine transformation, we may assume that $p=0$, $\Omega\subset \left\{z\in {\mathbb C}^n: z_n\neq0\right\}$, and the axis of $\Lambda_p$ is contained in the complex line $L:=\{z_1=\cdots=z_{n-1}=0\}$. Thanks to Corollary 5.2, we have a holomorphic function $h(z_n)$ on ${\mathbb C}_{z_n}^\ast$ such that $h'\neq 0$ everywhere, and $0$ is Picard point for $h$ on the planar domain
$\Lambda_p\cap L$.
Now we define a holomorphic map $F=(f_1,\cdots,f_n):\Omega\rightarrow {\mathbb C}^n$ as follows
$$
f_n(z)=h(z_n),\ \ \  f_k(z)=z_k/z_n^2,\ 1\le k\le n-1.
$$
Let $\zeta=(\zeta_1,\cdots,\zeta_n)\in ({\mathbb C}^\ast)^n$ be arbitrarily fixed.
By virtue of Corollary 5.2, we have a solution $z_n^\ast$ to the equation $h(z_n)=\zeta_n$ which can be arbitrarily close to $z_n=0$ inside the planar domain $\Lambda_p'\cap L$ where $\Lambda_p'\subset \Lambda_p$ is a smaller cone with the same axis.
Put $z_k^\ast=\zeta_k (z_n^\ast)^2$ for each $1\le k\le n-1$, and $z^\ast=(z_1^\ast,\cdots,z_n^\ast)$. Clearly, we have $F(z^\ast)=\zeta$ and $z^\ast\in \Lambda_p$
provided $|z_n^\ast|$ sufficiently small.  Thus $F$  maps $\Omega$ onto $({\mathbb C}^\ast)^n$. It is trivial to see that $F$ is locally biholomorphic. Combining this with Lemma 6.2, we conclude the proof.
\end{proof}

\begin{proposition} A domain $\Omega\subset {\mathbb C}^n$ admits a holomorphic map onto ${\mathbb C}^n$ which is locally biholomorphic if $\Omega$ belongs to one of the following domains:
\begin{enumerate}

 \item bounded domains with Lipschitz boundaries.

 \item convex domains.

 \item bounded homogeneous domains.

 \item model domains defined by
$$
\{(z,w)\in {\mathbb C}^n\times {\mathbb C}: {\rm Im\,}w>\psi(z)\},
$$
where $\psi\ge 0$ is continuous function on ${\mathbb C}^n$ so that $\psi(z)=O(|z|)$ holds near $z=0$.

\item Products $\Omega_1\times \Omega_2$, where $\Omega_1$ is a domain in ${\mathbb C}$ and $\Omega_2$ is a domain in ${\mathbb C}^{n-1}$.
\end{enumerate}
\end{proposition}

\begin{proof} In case $\Omega$ is a bounded domain with Lipschitz boundary,  we simply take the boundary point which is farthest from the origin so that the condition of the previous proposition is verified. Case $(2)$ and $(4)$ follow immediately from the previous proposition. Since each bounded homogeneous domain is biholomorphically equivalent to a Siegel domain of the second kind (cf. \cite{VinbergGindikin}), which is a convex (unbounded) domain, thus it admits a locally biholomorphic map onto ${\mathbb C}^n$. For the last case, we may assume that $\Omega_1\subset {\mathbb C}^\ast$ (note that the universal covering map $\pi:{\mathbb C}\rightarrow {\mathbb C}^\ast$ is surjective and locally biholomorphic). Take a point $p_1\in \partial \Omega_1$ so that $\Omega_1$ admits an inner ball at $p_1$. Without loss of generality, we may assume that the origin $0'\in \Omega_2$. Let $f_1$ be the holomorphic function given in Corollary 5.2. By an argument similar to the proof of the pervious proposition, we may take the desired holomorphic map from $\Omega$ onto ${\mathbb C}^\ast \times {\mathbb C}^{n-1}$ by
$$
(z_1,z_2,\cdots,z_n)\mapsto (f_1(z_1),z_2/(z_1-p_1)^2,\cdots,z_n/(z_1-p_1)^2).
$$
Combining this with Lemma 6.2, we conclude the proof.
\end{proof}

From the viewpoint of algebraic geometry, the most important noncompact complex manifolds are Zariski open sets in a projective algebraic manifold. Let $S=\bigcup_{j=1}^N S_j$ be an analytic hypersurface in some projective algebraic manifold $M$ with $S_j$ being irreducible. $S$ is said to be \emph{quasi-ample} if there exist positive integers $b_1,\cdots, b_N$, such that the effective Weil divisor $S_b=\sum_{j=1}^N b_jS_j$ is ample.

\begin{proposition} Let $M$ be a projective algebraic manifold of dimension $n$ and $S$ be a quasi-ample analytic hypersurface in $M$. Then $M\backslash S$ admits a nondegenerate holomorphic map onto ${\mathbb C}^n$.
\end{proposition}

\begin{proof} Let $L=[S_b]$ be the ample line bundle over $M$ associated to $S_b$ (multiplied by a sufficient large positive integer, we may assume that $L$ is very ample). Let $s$ be a global holomorphic section of $L$ whose associated divisor is precisely $S_b$. By Bertini's theorem, there are holomorphic sections $t_j\in H^0(M,L)\backslash \{0\}$, $1\le j\le n-1$, such that the associated hypersurfaces $T_j=\{t_j=0\}$ are smooth and $T_1,T_2,\cdots,T_{n-1},S$ are normal crossing at some smooth point $p$ of $S$. Let $R$ be the compact Riemann surface obtained by intersection of $T_j,\,1\le j\le n-1$ and put $R^\ast=R\backslash \{p\}$. Since $R^\ast$ an open Riemann surface (hence is Stein), we may choose a holomorphic function $f$ on $R^\ast$ such that it equals $0$ on a discrete sequence of points converging to $p$ and equals $1$ on another discrete sequence of points converging to $p$ by virtue of Cartan's Theorem $A$. Furthermore, we may choose $f$ so that $f'\neq 0$ at some point which is sufficiently close to $p$.  Thus $p$ is an essential singularity of $f$. Since $S_b$ is ample, $-\log|s|$ is a strictly plurisubharmonic exhaustion function of $M\backslash S$, so we see that $M\backslash S$ is a Stein manifold. Thus $f$ may be extended to a holomorphic function $\tilde{f}$ on $M\backslash S$. We define a holomorphic map $F=(f_1,f_2,\cdots,f_n):M\backslash S \rightarrow {\mathbb C}^n$ by
$$
f_n=\tilde{f},\ f_j=t_j/s,\ 1\le j\le n-1.
$$
Since $f$ omits at most one value, say $0$, in each deleted neighborhood of $p$ in $R$, we conclude that $F$ maps $M\backslash S$ onto ${\mathbb C}^\ast\times {\mathbb C}^{n-1}$. Furthermore, the complex Jacobian of $F$ is not identically zero. Combining with Lemma 6.2, we conclude the proof.
\end{proof}

\begin{corollary} Let $M$ be an Abelian variety and $S$ an ample divisor of $M$. Then there is a surjective nondegenerate holomorphic map from $M\backslash S$ to $M$.
\end{corollary}

\begin{proof} Let $\pi:{\mathbb C}^n\rightarrow M$ be the universal covering map. By the previous proposition, there is a surjective nondegenerate holomorphic map $F:M\backslash S \rightarrow {\mathbb C}^n$. It suffices to take $\pi\circ F$.
\end{proof}

Actually, every holomorphic map of an $n-$dimensional complex manifold ONTO ${\mathbb C}^n$ is nondegenerate, for the image of an everywhere degenerate holomorphic map is contained in a countable union of local complex varities in ${\mathbb C}^n$ of dimension $<n$.

\begin{proposition}
If $M$ is a Stein manifold of dimension $n$ and $S$ is an analytic hypersurface, then there exists a nondegenerate holomorphic map from $M\backslash S$ onto ${\mathbb C}^n$.
\end{proposition}

\begin{proof}
The argument is similar as Proposition 6.7. Take a holomorphic function $h$ on $M$ such that $S=h^{-1}(0)$. Fix a regular point $p$ of $S$. We may take holomorphic functions $f_1,f_2,\cdots,f_{n-1}$ on $M$ such that $f_j(p)=0$ and the holomorphic map $(f_1,\cdots,f_{n-1},h)$ is nondegenerate at $p$. Put $R=\{z\in M\backslash S:f_1(z)=\cdots=f_{n-1}(z)=0\}$. Since $R$ is a Stein space, we may choose a holomorphic function $f$ on $R$ such that it equals $0$ on a discrete sequence of points on $R$ converging to $p$ and equals $1$ on another discrete sequence of points on $R$ converging to $p$. Since $M\backslash S$ is Stein, $f$ can be extended to a holomorphic function $\tilde{f}$ on $M\backslash S$. It is easy to see that $F:=(f_1/h,\cdots,f_{n-1}/h,\tilde{f})$ maps $M\backslash S$ holomorphically onto ${\mathbb C}^{n-1}\times {\mathbb C}^\ast$.
\end{proof}

It seems quite possible that every $n-$dimensional Stein manifold admits a holomorphic map onto ${\mathbb C}^n$.

\section{A remark on manifolds with nonconstant bounded holomorphic functions}

Usually it is very difficult to know whether there exist nonconstant bounded holomorphic functions on a given complex manifold. Nevertheless, we have the following

\begin{proposition} {\it Let $M$ be a complex manifold which admits a nonconstant meromorphic function. Then there is an analytic hypersurface $S$ in $M$ such that the universal covering of $M\backslash S$ admits nonconstant bounded holomorphic functions. In particular, $M\backslash S$ is not universally dominated.}

\end{proposition}

We need the following

\begin{theorem}[Picard's Little Theorem] {\it Let $X$ be a complex manifold whose universal covering $\tilde{X}$ admits no nonconstant bounded holomorphic functions. Then every nonconstant holomorphic function on $X$ omits at most one value.}
\end{theorem}

\begin{proof} Suppose there is a nonconstant holomorphic function $f:X\rightarrow {\mathbb C}\backslash \{a,b\}$. Let
$\tilde{f}$ be a lift of $f$ to the universal coverings of $X$ and ${\mathbb C}\backslash\{a,b\}$, i.e., $\tilde{X}$ and
${\mathbb D}$. Then it is a nonconstant bounded holomorphic function on $\tilde{X}$. Contradiction.
\end{proof}

Although the proof is trivial, this result still has several amusing consequences. For instance, it follows that\/ {\it every}\/ nonconstant holomorphic function on a (Stein) quotient of ${\mathbb C}^n$ by a free, properly discontinuous group of automorphisms omits at most one value.

\begin{proof}[Proof of Proposition 7.1] We fix a nonconstant meromorphic function $f$ on $M$ and choose a cover $\{U_j\}$ of $M$ such that
 $$
 f=g_j/h_j
 $$
 on $U_j$, where $g_j, h_j$ are two relatively prime holomorphic functions. Thus
 $$
 f_{jk}=g_j/g_k=h_j/h_k
 $$
 is a non-vanishing holomorphic function on $U_j \cap U_k$. Thus the set
 $$
 S_0:=\bigcup\left\{z\in U_j:h_j(z)=0\right\}
 $$
 is an analytic hypersurface of $M$ and $f$ is a nonconstant holomorphic function on $M\backslash S_0$. Fix any
 two distinct complex numbers $a,b$ and put $S:=S_0 \cup f^{-1}(a)\cup f^{-1}(b)$. Thanks to Picard's Little Theorem,
 we see that the universal covering of $M\backslash S$ admits nonconstant bounded holomorphic functions.
 We claim that $M\backslash S$ is not universally dominated. Suppose on the contrary there is a dominant morphism
 $F:{\mathbb C}\rightarrow M\backslash S$. Let $\tilde{F}$ be a lift of $F$ to the universal covering $X$ of
 $M\backslash S$ and $h$ a nonconstant bounded holomorphic function on $X$. Then we get a nonconstant bounded
 holomorphic function $h\circ {\tilde{F}}$ on ${\mathbb C}$, which contradicts with Liouville's theorem.
 \end{proof}

\textbf{Acknowledgements.} We would like to thank Professor Franc Forstneri$\check{\rm c}$ for numerous comments on this paper. We also thank Dr. Qi'an Guan for catching an inaccuracy in the proof of Proposition 1.3. Finally, we wish to thank the referee for a very careful reading and many valuable comments.

\section{A supplement}
A trivial but interesting generalization of the classical Picard's little theorem to several complex variables is the following

     \begin{theorem}[Picard's little theorem]
     A holomorphic map $F:{\mathbb C}^n\rightarrow {\mathbb C}^m$ is non-degenerate in the sense that the image of $F$ is not contained in any analytic hypersurface of ${\mathbb C}^m$ if and only if for any nonconstant entire function $g$ on ${\mathbb C}^m$ the function $g\circ F$ assumes all complex numbers with at most one exception.
     \end{theorem}

     \begin{proof}
     Only if part: suppose there is a  nonconstant entire function $g$ of ${\mathbb C}^m$ so that the image of $g\circ F$ omits at least two points, say $0,1$ for the sake of simplicity. Let $\pi:{\mathbb D}\rightarrow {\mathbb C}\backslash \{0,1\}$ be the natural covering projection where ${\mathbb D}$ is the unit disc. Then there is a holomorphic map $G:{\mathbb C}^n\rightarrow {\mathbb D}$ satisfying $g\circ F=\pi\circ G$. By Liouville's theorem, $G$ has to be a constant $c$, so $g\circ F=\pi(c)$. Thus the image of $F$ has to lie in the analytic hypersurface $g^{-1}(\pi(c))$.

     If part: suppose the image of $F$ is contained in an analytic hypersurface $S$ of ${\mathbb C}^m$. It is well known that there is a nonconstant function $g$ on ${\mathbb C}^m$ such that $S=g^{-1}(0)$. Thus $g\circ F=0$.
     \end{proof}

      Unfortunately, it is not easy to formulate an\/ {\it analogous}\/ Picard's great theorem. Thus it is worthwhile to introduce a definition of Picard points for holomorphic maps. Let us first recall the following

      \begin{definition}
      Let $\Omega$ be a domain in ${\mathbb C}^n$. A point $p\in \partial \Omega$ is said to be a Picard point for some holomorphic function $f$ on $\Omega$ if $f$ assumes infinitely often in any neighborhood of $p$ all complex numbers with at most one exception.
      \end{definition}

       Forstneri$\check{\rm c}$ suggested a reasonable generalization of the concept for holomorphic maps as follows: $p$ is a Picard point for some holomorphic map $F:\Omega\rightarrow {\mathbb C}^m$ if it is a Picard point for any function of the form $P\circ F$ where $P$ is a nonconstant polynomial on ${\mathbb C}^m$. In view of Picard's little theorem mentioned above, it seems more natural to make the following

       \begin{definition}
      A point $p\in \partial \Omega$ is said to be a Picard point for some holomorphic map $F:\Omega\rightarrow {\mathbb C}^m$ if it is a Picard point for any function of form $g\circ F$ where $g$ is a nonconstant entire function on ${\mathbb C}^m$.
      \end{definition}

               \begin{lemma}
       Two definitions coincide when $m=1$.
        \end{lemma}

       \begin{proof}
        It suffices to show that the first definition implies the second. Let $U$ be any neighborhood of $p$. Then $f(\Omega\cap U)$ omits at most one complex number $a$. Let $g$ be a nonconstant entire function on ${\mathbb C}$. If $g$ is a polynomial, then $g\circ f(\Omega\cap U)$ omits at most one complex number $g(a)$. If $g$ is not a polynomial, then $g({\mathbb C})$ omits at most one complex number $b$. Thus for any $c\neq b$, $g^{-1}(c)$ contains at least two points $a_1,a_2$. We may assume $a_1\neq a$. Take $z^0\in \Omega\cap U$ such that $f(z^0)=a_1$, so $g\circ f(z^0)=c$.
        \end{proof}

      We may define ${\rm P}_{\alpha}(\Omega,{\mathbb C}^m)$, $\alpha>0$,  to be the set of holomorphic maps $F:\Omega\rightarrow {\mathbb C}^m$ whose set of Picard points coincides with $\partial \Omega$, and there are constants $C_1,C_2>0$ such that
$$
 |F(z)|\le C_1 e^{C_2 \hat{\delta}(z)^{-\alpha}}, \ \ \ z\in \Omega.
$$
The purpose of this supplement is to generalize Theorem 1.2 as follows

\begin{theorem} If\/ $\Omega\subset {\mathbb C}^n$ is a pseudoconvex domain, then\/ ${\rm P}_\alpha(\Omega,{\mathbb C}^m)$ is nonempty for any $\alpha>6$. Furthermore, if\/ $\Omega$ is bounded, then\/ ${\rm P}_\alpha(\Omega)$ is nonempty for any $\alpha>2$.
\end{theorem}

\begin{proof}
We follow closely the proof of Theorem 1.2. Choose first a dense sequence $\{p_j\}_{j=1}^\infty$ of points in $\partial \Omega$ and a mutually disjoint sequence of balls $\{B_j\}_{j=1}^\infty$ in $\Omega$ such that $\partial B_j $ touches $\partial \Omega$ only at $p_j$ for each $j$. By virtue of Lemma 5.5, we may choose in each $B_j$ a sequence $\{z^{j\mu}\}_{\mu=1}^\infty$ of points on the radius terminating at $p_j$
such that $\lim_{\mu\rightarrow \infty} z^{j\mu}=p_j$,
$\lim_{\mu\rightarrow \infty} k_{B_j}\left(z^{j\mu},z^{j(\mu+1)}\right)=0,
$
where $k_{B_j}$ denotes the Kobayashi distance of $B_j$,
and
$$
B\left(z^{j\mu},\frac12\hat{\delta}(z^{j\mu})|\log \hat{\delta}(z^{j\mu})|^{-1}\right)\subset B_j
$$
 are mutually disjoint for all $j,\mu$.
 Let $\{\zeta^k\}_{k=1}^\infty$ be a sequence of complex $m-$vectors which are dense in ${\mathbb C}^m$. Note that $\sum_{j,\mu=1}^\infty e^{-\hat{\delta}(z^{j\mu})^{-4}}<\infty$, in view of Lemma 2.4. We renumerate the sequence of positive even numbers by $\{\nu^1_{\mu}\}_{\mu=1}^\infty$ and take $\mu_{j1}>0$ such that
 $$
|\zeta^1|^2 \sum_{\mu=\mu_{j1}}^\infty  e^{-\hat{\delta}(z^{j\nu^1_\mu})^{-4}} <2^{-j}.
$$
For general $k>1$, we may choose $\mu_{jk}$ by induction such that
$$
|\zeta^k|^2 \sum_{\mu=\mu_{jk}}^\infty  e^{-\hat{\delta}(z^{j\nu^k_\mu})^{-4}}<2^{-jk}
$$
where $\{\nu_{\mu}^{k}\}_{\mu= 1}^\infty$ is a renumeration of the sequence $\{\nu_{2\mu}^{k-1}\}_{\mu= 1}^\infty$.

  Now define a sequence $\{c^{j\mu}\}_{j,\mu=1}^\infty$ of complex $m-$vectors by $c^{j\nu^k_\mu}=\zeta^k$ for all $j\ge 1$ and $\mu\ge \mu_{jk}$, and $c^{j\mu}=0$ otherwise. Clearly, we have
  $
  \sum_{j,\mu=1}^\infty |c^{j\mu}|^2 e^{-\hat{\delta}(z^{j\mu})^{-4}}<\infty,
  $
  so that
  $$
 \sum_{j,\mu=1}^\infty |c^{j\mu}|^2 e^{-a\,\hat{\delta}(z^{j\mu})^{-\alpha}}<\infty
 $$
 for any $a>0$ and $\alpha>4$. By virtue of Proposition 2.2/(2), there exists a holomorphic map $F:\Omega\rightarrow {\mathbb C}^m$ such that $F(z^{j\mu})=c^{j\mu}$ for all $j,\mu$, and
   $$
 |F(z)|\le C_1 e^{C_2 \hat{\delta}(z)^{-\alpha-2}}, \ \ \ z\in \Omega.
$$
Let $g$ be a nonconstant entire function on ${\mathbb C}^m$.  We claim that each $p_j$ is a Picard point of $g\circ F$.  Suppose on the contrary that $p_j$ is not a Picard point, then there would be a neighborhood $U_j$ of $p_j$ such that $g\circ F|_{\Omega\cap U_j}$ omits at least two values, in particular, it is a normal function on $B_j\cap U_j$. Applying Proposition 5.4 to the sequence $\{z^{j(2\mu-1)}\}_{\mu=1}^\infty$, we conclude that $g\circ F|_{B_j\cap U_j}$ has non-tangential limit $g(0)$ at $p_j$. But for a fixed complex $m-$vector $\zeta_0\notin g^{-1}(0)$, we may choose by diagonal process a subsequence $\{z^{j\mu_l}\}_{l=1}^\infty$ of $\{z^{j\mu}\}_{\mu=1}^\infty$ which tends to $p_j$, such that $F(z^{j\mu_l})\rightarrow \zeta_0$ as $l\rightarrow \infty$. It follows that $g\circ F(z^{j\mu_l})\rightarrow g(\zeta_0)\neq g(0)$, which is absurd.  The second conclusion may be proved similarly by using Proposition 2.3 in place of Proposition 2.2.
\end{proof}

\end{document}